\documentclass{article}%\documentclass[preprint]{amsart}
\usepackage{hyperref,amsmath,amsthm}
\usepackage{setspace}
\usepackage[toc,page]{appendix}
\usepackage{amssymb, amsmath, amsthm, graphics, }
\usepackage[dvips]{graphicx} 
\usepackage{xcolor, color}
\usepackage{caption}
\usepackage{hhline} 
\usepackage{adjustbox} %
\usepackage[utf8]{inputenc}
\usepackage{comment}
\excludecomment{comment}

\theoremstyle{plain}
\newtheorem{thm}{Theorem}[section]
\newtheorem{cor}[thm]{Corollary}
\newtheorem{lem}[thm]{Lemma}
\newtheorem{prop}[thm]{Proposition}

\theoremstyle{definition}
\newtheorem{example}[thm]{Example}
\newtheorem{dfn}[thm]{Definition}

\renewcommand{\qedsymbol}{\vrule height 5pt width 4pt depth -1pt} 

\def\ie{{i.e.,}\ }

\def\bigskip{\vspace{30pt}}

\newcommand{\p}{\varphi}
\renewcommand{\phi}{\varphi}
\newcommand{\CC}{\mathcal C}
\newcommand{\UX}{\underline X}
\newcommand{\UY}{\underline Y}
\def\hyperef#1{\hyperref[#1]{\ref{#1}}}

\newcommand{\RR}{\mathbb{R}}
\newcommand{\C}{\mathbb{C}}

\newcommand{\FF}{\mathbb{F}}

\newcommand{\R}{\mathcal{R}}
\newcommand{\M}{\mathcal{M}}
\newcommand{\N}{\mathcal{N}}
\newcommand{\Rp}{\mathbb{R}_{>0}}

\renewcommand{\S}{{S}^{1}}
\renewcommand{\l}{\lambda}
\newcommand{\rank}{\operatorname{rank}}
\newcommand{\CI}{\mathcal{C}_1}
\newcommand{\CII}{\mathcal{C}_2}

\newcommand{\ph}{\operatorname{ph}}
\newcommand{\im}{\operatorname{im}}

\newcommand{\backslant}[2]{{\raisebox{-.43em}{$#1$}\left\backslash \raisebox{.3em}{$#2$}\right.}}
\renewcommand{\iff}{\Leftrightarrow}
\renewcommand{\hat}{\widehat}
\renewcommand{\tilde}{\widetilde}
\newcommand{\MM}{\underline{\M}}
\newcommand{\G}{\mathcal G}
\usepackage{calc}
\newcommand{\bp}{
  \mathop{
    \vphantom{\bigoplus} 
    \mathchoice
      {\vcenter{\hbox{\resizebox{\widthof{$\displaystyle\bigoplus$}}{!}{$\boxplus$}}}}
      {\vcenter{\hbox{\resizebox{\widthof{$\bigoplus$}}{!}{$\boxplus$}}}}
      {\vcenter{\hbox{\resizebox{\widthof{$\scriptstyle\oplus$}}{!}{$\boxplus$}}}}
      {\vcenter{\hbox{\resizebox{\widthof{$\scriptscriptstyle\oplus$}}{!}{$\boxplus$}}}}
  }\displaylimits 
}
\begin{document}

\title{\Large REALIZATION SPACES OF UNIFORM PHASED MATROIDS}

\author{Amanda Ruiz\\\\Department of Mathematics, \\University of San Diego, \\SAN DIEGO, CA 92110\\\\alruiz@sandiego.edu}
%\address{Department of Mathematics, University of San Diego, SAN DIEGO, CA 92110}
%\email{alruiz@sandiego.edu}

%\keywords{Matroid, oriented matroid, phased matroid, complex matroid, phirotope, realization space, realizability}

%\subjclass[2010]{Primary 52C40 ; Secondary 52B40.}

% 05A19 Combinatorial identities, bijective combinatorics
% 52C35 Arrangements of points, flats, hyperplanes

\date{}

%\thanks{
%A portion of this work was completed while AR was a postdoctoral fellow at Harvey Mudd College and was supported in part by the National Science Foundation under Grant DMS-0839966.  BU?}

\maketitle% makes the page numbers roman numerals, doesn't count
% these pages in the table of contents
%\frontmatter
%
%\newpage

\begin{abstract}
{
A phased matroid is a matroid with additional structure which plays the same role for complex vector arrangements that oriented matroids play for real vector arrangements.  

The realization space of an oriented (resp., phased) matroid is the space  of vector arrangements in $\RR^n$ (resp., $\C^n$) that correspond to oriented (resp., phased) matroid, modulo a change of coordinates.  According to Mn\"ev's Universality Theorem, the realization spaces of uniform oriented matroids with rank greater than or equal to $3$ can be as complicated as any open semi-algebraic variety. 
   
 In contrast, uniform phased matroids which are not essentially oriented have remarkably simple realization spaces if they are uniform. 
 
We also present a criterion for realizability of uniform phased matroids that are not essentially oriented.

}
\end{abstract}

%\newpage

%\tableofcontents

%
\section{Introduction}

A matroid is a combinatorial object that abstracts the notion of linear independence in a vector configuration over an arbitrary field $\FF$.  Oriented matroids are matroids with additional structure which encodes more geometric information than matroids by assigning to each ordered basis an element of $\{-1,0,+1\}$.   The theories of matroids and oriented matroids are major branches of combinatorics with applications in many fields of mathematics, including topology, algebra, graph theory, and geometry.   Their contribution to mathematics motivates the question \emph{what is a complex matroid?} (In fact,  Ziegler wrote a paper with this title \cite{ziegler}).

Several mathematicians have addressed the question posed by Ziegler, with varying answers. In \cite{ziegler}, Ziegler  proposed his own definition of a complex matroid. His complex matroids are discrete objects with desirable characteristics consistent with being a generalization of oriented matroids, but they lack a notion of cryptomorphisms, an important characteristic of matroids and oriented matroids. 

Another approach to Ziegler's question is to generalize oriented matroids by assigning to each ordered basis an element of the set $\S\cup\{0\}$. 
In  \cite{dressw}, Dress and Wenzel show that  this non-discrete approach corresponds to basis ``orientations" over the fuzzy ring $\C//\Rp$, of which $\S \cup \{0\} $ is a subset.  
Phased matroids (referred to as \emph{complex matroids} in \cite{andersondel,BKR-G}) are an answer developed by Anderson and Delucchi  \cite{andersondel}, growing out of work by Below, Krummeck and Richter-Gebert \cite{BKR-G} and building on Delucchi's diploma thesis\cite{Delucchi}. 
In \cite{andersondel}, Anderson and Delucchi explore how much of the foundations of oriented matroids can be paralleled with the structure of the set $\S\cup\{0\}$.  They give a phased analog to the chirotope and circuit axioms of oriented matroids, and show that the definitions are cryptomorphic. A recent paper of Baker and Bowler \cite{baker}, yet to appear as of the writing of this introduction, generalizes oriented matroid and phased matroid properties into matroids over hyperfields, proving the definitions are cryptomorphic on the hyperfield level, and uniting the fields of oriented matroids, phased matroids, and matroids over hyperfields other than signs of real numbers and phases of complex number. 

Baker and Bowler's paper is already expected to be important by allowing common properties of phased matroids and oriented matroids to be proven once, on the hyperfield level.

In this paper, we will focus on some of the differences between phased matroids and oriented matroids that can not be proven on at the hyperfield level, but rely on the extra degree of freedom allowed over the complex numbers that is not allowed for over the real numbers.  Some long term goals of phased matroid theory are to apply it to complex vector bundles and complex hyperplane arrangements in a way analogous  to the applications of oriented matroids to real hyperplane arrangements and real vector bundles.

A celebrated theorem in oriented matroid theory is Mn\"ev's Universality Theorem \cite{mnev}, which describes how complicated the topology of realization spaces of oriented matroids (including uniform oriented matroids) can be. An obvious question for phased matroid theory,  is: How does the topology of  realization spaces of phased matroids compare to that of oriented matroids?

This paper sets out to answer that question in the case of uniform phased matroids that are not essentially oriented. 
Surprisingly, the realization spaces of uniform phased matroids that are not essentially oriented are remarkably simple. In fact, regardless of the rank of a uniform phased matroid with a groundset of size $n$, if the phased matroid is not essentially oriented then the realization space is homeomorphic to $\Rp^{n-1}$ (Theorem \ref{uniformthm}).

%We also show that rank $3$ phased matroids on $n$ elements which are  not essentially oriented have a realization space homeomorphic to  $\Rp^k$.  Theorem \hyperef{rank3thm} will give $k$ in terms of $n$ and the structure of the  matroid.  

Determining realizability of an oriented matroid is NP hard (\cite{redbook}, Theorem 8.7.2). In contrast, in Theorem \ref{realizability} we give a polynomial time algorithm for determining realizability of a uniform phased matroid that is not essentially oriented. This algorithm also, in the affirmative case, gives a unique canonical realization of the phased matroid. 
 
% Some phased matroids take on the  realization spaces of oriented matroids. In Section \ref{seperable}, we define some constructions of phased matroids along with their resulting realization spaces. 

%%%%%%%%%%%%%%%%%%%%%%%%%%%%%%%%%%%%%%%%%%%%%%%%%%%%%%%%%%%%%%%%%%%%%%%%%%%%%%%%%%%%%%%%%%%%%%%%%%%%%%%%%%%%%%%%%%%%%%%%%%%%%%%%%%%%%%%%%%%%%%%%%%%%%%%%%%%%%%%%%%%%%%%%%%%%%%%%%%%%%%%%%%%%%%%%%%%%%%%%%%%%%%%%%%%%%%%%%%%%%%%%%%%%%%%%%%%%%%%%%%%%%%%%%%%%%%%%%%%%%%%%%%%%%%%%%%%%%%%%%%%%%%%%%%%%%%%%%%%%%%%%%%%%%%%%%%%%%%%%%%%%%%%%%%%%%%%%%%%%%%%%%%%%%%%%%%%%%%%%%%%%%%%%%%%%%%%%%%%%%%%%%%%%%%%%%%%%%%%%%%%%%%%%%%%%%%%%%%%%%%%%%%%%%%%%%%%%%%%%%%%%%%%%%%%%%%%%%%%%%%%%%%%%%%%%%%%%%%%%%%%%%%%%%%%%%%%%%%%%%%%%%%%%%%%%%%%%%%%%%%%%%%%%%%%%%%%%%%%%%%%%%%%%%%%%%%%%%%%%%%%%%%%%%%%%%%%%%%%%%%%%%%%%%%%%%%%%%%%%%%%%%%%%%%%%%%%%%%%%%%%%%%%%%%%%%%%%%%%%%%%%%%

\section{Phased matroid definitions}

We  will use polar coordinates to  denote nonzero complex numbers. Thus we will view $\C\backslash\{0\}$ as $\Rp\times \S$, where $\S=\{z\in\C\mid |z|=1\}$ is the complex unit circle.

\begin{dfn}[Phase of a complex number~\cite{BKR-G}]  Let  $s\in \Rp,$ and $ \alpha\in\S$. For the  complex number $z=s\alpha\in\C$ the \emph{phase} of $z$ is  
$$\ph(z)=\begin{cases} 0\mbox{ if } s=0\\  \alpha\mbox{ if } s\ne0.\end{cases}$$ Thus $\ph(z) \in \S\cup\{0\}$.
Also, $s$ is the \emph{norm} of $z$.
\end{dfn}

\begin{dfn}[Hypersum \cite{baker}] Define the \emph{hypersum}  $\boxplus(S)=\displaystyle\bp_{k=1}^{m}\alpha_k$ of a finite set $S=\{\alpha_1,\ldots,\alpha_m\} \subset S^1 \cup \{0\}$ to
be the set of all phases of strictly positive linear combinations of $S$. Thus
 \begin{itemize}
\item $\boxplus(\emptyset) =\emptyset$.
\item $\boxplus(\{\mu\}) = \{\mu\} $ for all $\mu$.
\item $\boxplus(\{\mu, -\mu\}) =  \boxplus(\{ \mu,0, -\mu\})=\{\mu,0,  -\mu\}$ for all $\mu$.
\item if $S=\{e^{i\alpha_1},\ldots,e^{i\alpha_k}\}$ with $k\ge2$ and $\alpha_1 <\cdots<\alpha_k<\alpha_1+\pi$   then
$\boxplus(S)=\boxplus(S\cup\{0\})=\{e^{i\lambda} :\alpha_1 <\lambda<\alpha_k\}$.
\item if $S=\{e^{i\alpha_1},\ldots,e^{i\alpha_k}\}$ with $k\ge3$ and $\alpha_1<\cdots<\alpha_k=\alpha_1+\pi$, then $\boxplus(S)=\boxplus(S\cup\{0\})=\{e^{i\lambda}:\alpha_1<\lambda<\alpha_k\}$. 
\item otherwise (\ie if the nonzero elements of $S$ do not lie in a closed half-circle of $S^1$), $\boxplus(S) = S^1 \cup \{0\}.$
\end{itemize}
\end{dfn}
Note that in \cite{andersondel} the hypersum is reffered to as \emph{phased convex hull}.
\begin{dfn}[Phirotope \cite{andersondel}]\label{phirodef} A rank $r$ \emph{phirotope} on the set $E$  is a non-zero, alternating function $\p:E^r\rightarrow \S\cup \{0\}$ satisfying the following \emph{combinatorial complex Grassmann--Pl\"ucker relations}.
\begin{quote}For any two subsets $\{x_1,\ldots,x_{r+1}\},\{y_1,\ldots,y_{r-1}\}$ of $E$, 
\begin{equation}\label{gprels}
0\in\displaystyle\bp_{k=1}^{r+1}(-1)^k\p(x_1,\ldots,\hat{x}_{k},\ldots,x_{r+1})\cdot\p(x_k,y_1,\ldots,y_{r-1}).\end{equation}\end{quote}

 \end{dfn}
 
 The phirotope definition is a generalization of a \emph{chirotope} in which $\im(\phi)\subseteq\{-1,0,+1\}$. 

As suggested by the name, the combinatorial complex Grassmann-Pl\"ucker relations are motivated by the Grassmann-Plucker relations, which the minors of every matrix in $C^{r\times n}$ with $r\le n$ satisfy. So every such matrix gives rise to a phirotope. 
 
 \begin{example}\label{runex} Consider the $3\times 5$ matrix \[M= \left(
\begin{array}{ccccccc}
    1& {}  0& 0& \frac{1}{2}e^{i\frac{\pi}{4}}&{\frac{1}{3}e^{i\frac{\pi}{2}}}  \\
 0 & {1}  & 0& {1} &  { \frac{4}{3}e^{i\frac{\pi}{4}}}   \\
 0 & {0}  & 1 & {0} &  {{-1}}  
\end{array}
\right)\] with complex entries.
The function $\p_M:[5]^3\rightarrow S^1\cup\{0\}$ such that $\p_M(i,j,k)=\ph(\det(M_{i,j,k}))$ satisfies the combinatorial complex Grassmann-Pl\"ucker relations and is a phirotope. There is a phirotope for any $n\times r$ matrix with $n\ge r$.  
 \end{example}
 
 \begin{prop}[Phirotope of a matrix \cite{BKR-G}]\label{phiromatrix}Every matrix $M\in\C^{r\times |E|}$ with $|E|\ge r$ gives rise to the following function:
\[\begin{array}{rccl}
\p_M:&E^r&\rightarrow&\S\cup\{0\}\\
&\l=(\lambda_1,\ldots,\lambda_r)&\mapsto&\ph([{\lambda}]_M).\end{array}\]  Furtheremore, $\p_M$ is a phirotope.
\end{prop}
 
 Notice that if $M\in\RR^{r\times|E|}$, then $\phi_M$ is a chirotope.
 
 Consider $M\in\C^{n\times r}$. Let $A\in GL(r,\C)$. Then $\p_{AM}=\ph(\det(A))\p_{M}$. So multiplication on the left by $GL(r,\C)$ rotates the coordinate system.

\begin{dfn}[Phased matroid \cite{andersondel}]\label{PMdefphiro} The set  $\M:=\{\alpha\p\mid \alpha\in\S\}$ is a \emph{phased matroid}. \end{dfn} 

 If $\p$ is a chirotope, then $\{-\p,\p\}$ is an \emph{oriented matroid.}

\begin{dfn}[Realizable phased matroid, realization of a phased matroid]
 Let $\M$ be a phased matroid with  a phirotope $\p$. $\M$ is \emph{realizable} if there exists a matrix $M$ such that $\p_M=\p$. The matrix $M$ is a \emph{realization} of $\M$.  We can also say that the phirotope $\p$ is realizable with realization $M$.
\end{dfn} 
Define the \emph{equivalence class of a matrix }as $[M]:=\{AM|A\in GL(r,\C)\}$. Notice that if $M$ is a realization of $\M$, then every element in the equivalence class $[M]$ is a realization of $\M$. 

\begin{example} The matrix $M$ in Example \ref{runex} is a realization of the realizable phased matroid $\{\alpha\p_M\mid \alpha\in\S\}$ denoted $\M_M$. Also, for any invertible $3\times 3$ matrix $A$, $AM$ is a realization of $\M_M$.
\end{example}

\begin{dfn}[Underlying matroid of a phased matroid \cite{andersondel}] Let $\M$ be a rank $r$ phased matroid with phirotope $\p$. Then $\MM:= \{\{i_1,\ldots,\i_r\}|\p(i_1,\ldots,i_r)\ne0\}$ is a matroid called the underlying matroid of $\M$. 
\end{dfn}

If $\M$ is realizable with realization $M$, then its underlying matroid $\MM$ is realizable with realization $M$.

\begin{example} The underlying matroid of $\underline \M_M$ obtained from the matrix $M$ in Example \ref{runex} is 
\[\left\{\{1,2,3\}, ~\{1,2,5\}, ~\{1,3,4\}, \{1,3,5\},~\{1,4,5\}, ~\{2,3,5\}, ~\{3,4,5\}\right\}.
\]
\end{example}

 %%%%%%%%%%%%%%%%%%%%%%%%%%%%%%%%%%%%%%%%%%%%%%%%%%%%%%%%%%%%%%%%%%%%%%%%%%%%%%%%%%%%%%%%%%%%%%%%%%%%%%%%%%%%%%%%%%%%%%%%%%%%%%%%%%%%%%%%%%%%%%%%%%%%%%%%%%%%%%%%%%%%%%%%%%%%%%%%%%%%%%%%%%%%%%%%%%%%%%%%%%%%%%%%%%%%%%%%%%%%%%%%%%%%%%%%%%%%%%%%%%%%%%%%%%
 
 \subsection{Remark about uniform phased matroids}
 
 Let $\p$ be a phirotope. There is an underlying matroid, $\underline{\M}$, with bases $\mathcal B:=\{\lambda|\p(\lambda)\ne0\} $.  If $\im(\p)\subseteq\S$, then $\p$ and $\MM$ are \emph{uniform}.
 
 The matrix $M$ in Example \ref{runex} gives rise to a realizable phased matroid $\M(M)$ which is not uniform because $\p(1,2,4)=0$. 
 
 %%%%%%%%%%%%%%%%%%%%%%%%%%%%%%%%%%%%%%%%%%%%%%%%%%%%%%%%%%%%%%%%%%%%%%%%%%%%%%%%%%%%%%%%%%%%%%%%%%%%%%%%%%%%%%%%%%%%%%%%%%%%%%%%%%%%%%%%%%%%%%%%%%%%%%%%%%%%%%%%%%%%%%%%%%%%%%%%%%%%%%%%%%%%%%%%%%%%%%%%%%%%%%%%%%%%%%%%%%%%%%%%%%%%%%%%%%%%%%%%%%%%%%%%%%%%%%%%%%%%%%%%%%%%%%%%%%%%%%%%%%%%%%%%%%%%%%%%%%%%%%%%%%%%%%%%%%%%%%%%%%%%%%%%%%%%%%%%%%%%%%%%%%%%%%%%%%%%%%%%%%%%%%%%%%%%%%%%%%%%%%%%%%%%%%%%%%%%%%%%%%%%%%%%%%%%%%%%%%%%%%%%%%%%%%%%%%%%%%%%%%%%%%%%%%%%%%%%%%%%%%%%%%%%%%%%%%%%%%%%%%%%%%%%%%%%%%%%%%%%%%%%%%%%%%%%%%%%%%%%%%%%%%%%%%%%%%%%%%%%%%%%%%%%%%%%%%%%%%%%%%%%%%%%%%%%%%%%%%%%%%%%%%%%%%%%%%%%%%%%%%%%%%%%%%%%%%%%%%%%%%%%%%%%%%%%%%%%%%%%%%%%%%
 
 \section{The realization space}
 
 \begin{dfn}[Realization space]\label{Rspace} The \emph{realization space} over $\C$ of a rank $r$ phased matroid $\M$ on the ground set $E$   
is the quotient space 
$$\mathcal{R}_\C(\M)=
\backslant{GL(r,\C)}{\{M\in \C^{r\times |E|}\mid \text{ M is a realization of } \M\}}.$$

If $\M$ is an oriented matroid, 
$$\mathcal{R}_\RR(\M)=
\backslant{GL(d,\RR)}{\{M\in \RR^{r\times |E|}\mid \text{ M is a realization of } \M\}}.$$
\end{dfn}

For the special case of oriented matroids, realization spaces  have been well studied. Results from oriented matroid theory apply to the phased matroids that are also oriented matroids.

\begin{thm}[Mn\"ev's Universality Theorem \cite{redbook}, Theorem 8.6.6]\label{mnevs}\text{ }
\begin{enumerate}
 \item Let $V\subseteq\RR^s$ be any semialgebraic variety. Then there exists a rank $3$ oriented matroid $\M$ whose realization space $\R(\M)$ is homotopy equivalent to $V$.
\item If $V$ is an open subset of $\RR^s$, then $\M$ may be chosen to be uniform.
\end{enumerate}
\end{thm}

However, when we consider uniform phased matroids that are not essentially oriented, the topology of the realization space is surprisingly simple. 
\begin{thm}\label{uniformthm}
Let $\M$ be a rank $r$, uniform, not essentially oriented, realizable phased matroid on $[n]$. Then $\R(\M)\cong\Rp^{n-1}$. \end{thm}

The proof of Theorem \ref{uniformthm} will appear in Section \ref{sec:proof}. As a motivating Lemma, we will prove the rank 2 case (in which we can drop conditions of uniformality and not essentially orientetability.)
\begin{lem}\label{rigid}
Let $\M$ be a simple, rank $2$, realizable phased matroid on $n$ elements. Then $\R(\M)\cong\Rp^{n-1+k}$ where  $k=0$ if $\M$ is not essentially oriented and $k>0$ if $\M$ is essentially oriented. \end{lem}

{This result, for rank $2$ uniform  phased matroids that are not essentially oriented,  was previously proven in \cite{BKR-G} using cross ratios. Since cross ratios do not generalize to higher dimensions, their proof is not generalizable to phased matroids with  rank $> 2$. For essentially oriented phased matroids the Lemma follows from the fact that all rank-2 oriented matroids have contractible realization spaces \cite{redbook}. We will show an alternative proof for the not essentially oriented case in Section \ref{sec:proof}.
 }
 
 %%%%%%%%%%%%%%%%%%%%%%%%%%%%%%%%%%%%%%%%%%%%%%%%%%%%%%%%%%%%%%%%%%%%%%%%%%%%%%%%%%%%%%%%%%%%%%%%%%%%%%%%%%%%%%%%%%%%%%%%%%%%%%%%%%%%%%%%%%%%%%%%%%%%%%%%%%%%%%%%%%%%%%%%%%%%%%%%%%%%%%%%%%%%%%%%%%%%%%%%%%%%%%%%
%{ We can also drop the condition of uniformality for rank 3 phased matroids. 
% \begin{thm}\label{rank3thm} Let $\M$ be a  rank-$3$, not essentially oriented phased matroid on $[n]$. Then $\R(\M)\cong \Rp^{n-1+k}$ where $k=0$ if $\M$ is not a direct sum, parallel connection, or series connection of non-trivial phased matroids, and $k>0$ if $\M=P(\M_1,\M_2)$ where $\M_1$ or $\M_2$ is essentially oriented. 
%\end{thm}
%
%Theorem \ref{rank3thm} will be proved in Section \ref{sec:3proof}. }
%%%%%%%%%%%%%%%%%%%%%%%%%%%%%%%%%%%%%%%%%%%%%%%%%%%%%%%%%%%%%%%%%%%%%%%%%%%%%%%%%%%%%%%%%%%%%%%%%%%%%%%%%%%%%%%%%%%%%%%%%%%%%%%%%%%%%%%%%%%%%%%%%%%%%%%%%%%%%%%%%%%%%%%%%%%%%%%%%%%%%%%%%%%%%%%%%%%%%%%%%%%%%%%
First, in the following section, we will develop   tools which will be useful in proving the above statements.

 %%%%%%%%%%%%%%%%%%%%%%%%%%%%%%%%%%%%%%%%%%%%%%%%%%%%%%%%%%%%%%%%%%%%%%%%%%%%%%%%%%%%%%%%%%%%%%%%%%%%%%%%%%%%%%%%%%%%%%%%%%%%%%%%%%%%%%%%%%%%%%%%%%%%%%%%%%%%%%%%%%%%%%%%%%%%%%%%%%%%%%%%%%%%%%%%%%%%%%%%%%%%%%%%%%%%%%%%%%%%%%%%%%%%%%%%%%%%%%%%%%%%%%%%%%%%%%%%%%%%%%%%%%%%%%%%%%%%%%%%%%%%%%%%%%%%%%%%%%%%%%%%%%%%%%%%%%%%%%%%%%%%%%%%%%%%%%%%%%%%%%%%%%%%%%%%%%%%%%%%%%%%%%%%%%%%%%%%%%%%%%%%%%%%%%%%%%%%%%%%%%%%%%%%%%%%%%%%%%%%%%%%%%%%%%%%%%%%%%%%%%%%%%%%%%%%%%%%%%%%%%%%%%%%%%%%%%%%%%%%%%%%%%%%%%%%%%%%%%%%%%%%%%%%%%%%%%%%%%%%%%%%%%%%%%%%%%%%%%%%%%

\section{Useful tools}
In this section, we will introduce some tools that help prove Theorem \ref{uniformthm}.  We will see how some of the content and structure of a potential realization is determined by the phirotope. In fact,  in Section \ref{canonicalformsec} we will be on the hunt for a \emph{canonical form} of a realization of a phased matroid. In Section \ref{dellemma}, through the Triangle Lemma we see the geometry of the complex space that real space does not have. This provides  insight into why the essentially oriented phased matroids have realization spaces so different from the non-essentially oriented phased matroids. The definitions in this section refer to uniform phased matroids, but they can all be generalized to the non-uniform case. See \cite{mydissertation} for details. 

\subsection{Canonical Form}\label{canonicalformsec}
Every phased matroid $\M$ can be compared to a (possibly different) phased matroid $\M'$ which is in \emph{canonical form} (Definition \ref{canonicarealiz}). While $\M\ne\M'$, we will show that $\R(\M)\cong\R(\M')$. So we can look at the realization space of $\M'$ in order to understand the realization space of $\M$. We will find that a phased matroid in canonical form provides advantages such as easily determining whether or not the phased matroid is essentially oriented, as well as easily building a realization of a realizable phased matroid that is not essentially oriented, and computing the realization space.

In order to get a phased matroid into canonical form, there are a lot of details that we need to address. In this section we will use matroid theoretic facts about bipartite graphs, compute signs of some special permutations, and recall some basic geometry facts about similar triangles. 

The next lemma will help us prove Corollary \ref{phaseminor}, which is an important part of the process of building up a realization of a given phased matroid  containing as much information about the potential realization as we can from a phirotope. In fact, Corollary \ref{phaseminor} simply allows us to relate the minor of a matrix $(I|N)$ containing only columns of $N$, to a maximal minor of $(I|N)$, by bringing in columns of $I$. The two minors will differ only by a sign, but how the sign is determined is dependent on so many things, that it is worth it's own lemma.

%
%Let $H=\{h_1,...,h_k\}\subset [r],$ and $A_1,...,A_k>r$
%
%Let the set $\hat{H_r}$ be $\{1,...\hat{h_1},...,\hat{h_k},...r\}$
%
%The permutation $(1,...,{h_1}-1, A_1, {h_1}+1,...,{h_2}-1, A_2, {h_2}+1,...,{h_k}-1, A_k, {h_k}+1,...r)$
%
%has sign $\sum_{h\in H}\sum_{k\in\hat H_r,k>h}1$
%

Denote by $\hat H_r=(1,\ldots,\hat h_1,\ldots,\hat h_k,\ldots,r),$ the ordered set $(1,\ldots,r)$ without the elements $\{h_1,\ldots,h_k\}=H_r$.

Let $J=(j_1,\ldots j_k)$ where $j_l\in [r+1,\ldots,n]$ and $j_l<j_{l+1}$ for all $1\le l\le k$.  Denote by $\pi(\hat H_r,J)$ the permutation of ${\hat H_r,J}$ where the elements of $J$ replace (in ascending order) the elements of $H_r$ which are missing from $\hat{H_r}$. 

\begin{lem}\label{minorsign} Consider the matrix $(I|N)$. Let $N'$ be a $k\times k$ submatrix of $N$ consisting of rows $H=\{h_1,\ldots,h_k\}\subset[r]$
% How is this H related to the H_r above?
 and columns $J=\{j_1,\ldots,j_k\}\subset[n]$. Let $\sigma=
%dk+\frac{k(k+1)}{2}-k^2-\Sigma_{i\in I_r}i
\sum_{h\in H}\sum_{k\in \hat H_r,k>h}1$. Then $\det(N')=(-1)^\sigma\det(I|N)_{\hat H_r,J}$. 
\end{lem}

\begin{proof} Note that $\sigma$ is summing, for each $h\in H$, the number of elements in $[r]$ %  I think I mean [k] no I think i mean [r]
but not in $H$ %what the fuck is I? I is the identity matrix. what is an element of I?
which are greater than $i$. 
%From what I recall, what I want to say here is that I want to count, for each column missing from the identity matrix, the number of columns after it which ARE included, so as the columns from N have to jump over it to be shuffled in. 
The sign $(-1)^\sigma$ is the sign of the shuffle permutation, $S(H,J)$, that replaces, in order, the elements of $J$ for the elements of $H$ removed from $[r]$. For example, let $r=7$, $H=\{2,3,5\}$, and $J=\{8,11,12\}$. Then the permutation  is $(1,8,11,4,12,6,7)$.  

Since the determinant is an alternating function, the determinant  of the submatrix $N$ is given by $$\det(N)=\det((I|N)_{S(H,J)})=(-1)^\sigma\det((I|N)_{\hat H_r,J}).$$ \end{proof}

 \begin{cor}\label{phaseminor} Let $\M$ be a rank $r$ realizable phased matroid with  realization $(I|N)$ and phirotope $\p=\p_{(I|N)}.$ Let $N'$ be a  $k\times k$ submatrix of $N$ consisting of rows $H=\{h_1,\ldots,h_k\}$ and columns $J=\{j_1,\ldots,j_k\}$.   Then, for $\sigma$ defined as in Lemma \ref{minorsign}, $\ph(\det(N'))=(-1)^\sigma\p(\hat H_r,J)$. 
\end{cor}
\begin{proof}   Since $\M$ is realizable, $$\p(\hat H_r,J)=\ph(\det((I|N)_{\hat H_r,J}))=(-1)^\sigma\ph(\det(N')).$$ So this follows from Lemma \ref{minorsign}. 
\end{proof}

Throughout this paper we are only interested in the case when $k\in\{1,2\}$. In the case that $k=1$ we get the phases of the entries of any potential realization of $\p$ because 
\[(-1)^{r-i}\ph((I|N)_{i.j})=\ph(\det((I|N)_{\{1,\ldots,\hat i,\ldots,r,j\}})=\p(1,\ldots,\hat i,\ldots,r,j).\] 
\begin{cor}\label{phirophase} Suppose $\{1,\ldots,r\}$ is a basis of $\M$. A phirotope of $\M$ determines  the phases of $N$ for any realization $(I|N)$ of $\M$.  \end{cor}

For example, if we consider the matrix $M$ from Example \ref{runex}, we see that \[\p(1,3,5)=\ph\left(\det \left(
\begin{array}{ccc}
    1&  0&{\frac{1}{3}e^{i\frac{\pi}{2}}}  \\
 0 &  0&   { \frac{4}{3}e^{i\frac{\pi}{4}}}   \\
 0 &  1 &  {{-1}}  
\end{array}
\right)\right)=-e^{i\frac{\pi}{2}}=(-1)^{3-2}\ph(M_{2,5})\]
%%%%

\begin{dfn}[Rephasing of a phirotope \cite{BKR-G}]\label{rephase} Let $\p$ be a rank $r$ phirotope of a phased matroid $\M$.
 Let $\rho=(\rho_1,\ldots,\rho_n)\in(\S)^n$. 
The function 
\[\begin{array}{rccl}
\p^\rho:&[n]^r&\rightarrow&\S\cup\{0\}\\
&(\lambda_1,\ldots,\lambda_r)&\mapsto&\rho_{\lambda_1}\ldots\rho_{\l_r}\p(\lambda_1,\ldots,\lambda_r)\end{array}\] is the \emph{rephasing of $\p$ by $\rho$.} 
\end{dfn}

In \cite{BKR-G} rephasing is refered to as \emph{reorienting}, motivated by it's similarity to reorientation of oriented matroids. 

\begin{lem}[\cite{BKR-G}] \label{rephasephiro} The function $\p^\rho$ is a phirotope. 
\end{lem}

 Let $\M^\rho$ be the phased matroid with phirotope $\p^\rho$. Then $\M^\rho$ is the \emph{rephasing} of $\M$ by $\rho$. 
Notice that if we start with a chirotope $\chi$, and apply rephase by $\rho\in(\S)^n$ then $\im(\chi^\p)\subseteq\{-\alpha_0,0,\alpha_0\}$ for some $\alpha_0\in \S$. So $\chi^\p$ is a phirotope, but not a chirotope, and $\{\alpha\chi^\p|\alpha\in\S\}$ is a phased matroid, but not an oriented matroid.

Let $D(x_1,\ldots,x_m)$ be the $m\times m$ diagonal matrix with entries $x_1,\ldots,x_m$ on the diagonal.
\begin{lem}[\cite{BKR-G}]\label{rephaserealiz} If $M$ is a realization of $\M$  then
$M\cdot D(\rho_1,\ldots\rho_n)$ is a realization of $\M^{\rho}$.
\end{lem}

The following notion of scaling equivalence is borrowed from matroid theory.

\begin{dfn}[Scaling equivalent matrices \cite{oxley}]
Let $M$ and $N$ be $r\times n$ matrices over a field $\mathbb F$. We say $M$ and $N$ are \emph{scaling equivalent} if
 $N$ can be obtained by scaling rows and columns of $M$ by non-zero elements of $\mathbb F$. 

\end{dfn}

Scaling rows and columns of a $r\times n$ matrix $M$ is algebraically equivalent to multiplying $M$ on the left by a diagonal $r\times d$ matrix and on the right by an  $n\times n$ diagonal matrix (both with non-zero determinant). 

In matroid theory, if two matrices are scaling equivalent, then they give rise to the same matroid. This is not the case with phased matroids. However, scaling equivalent matrices give rise to phased matroids that are rephasings of each other. The following result demonstrates why these are useful to us when studying the realization space. 

\begin{thm}\label{equivrealspace} Let $\M,~\mathcal N$  be phased matroids. Let $M\in\R(\M)$, $N\in\R(\mathcal N)$. If $M$ and $N$ are scaling equivalent, then $\R(\M)\cong \R(\mathcal N)$.
\end{thm} 

\begin{proof}
Since $M$ and $N$ are scaling equivalent matrices, there exist diagonal matrices $D_1$ and $D_2$ such that $M=D_1ND_2$. Consider the map $f:\R(\M)\rightarrow \R(\mathcal N)$ where $f([H])=[D_1HD_2]$. Let $K\in[H]$. Then there is an $A\in GL(r,\C)$ such that $K=AH$.  

We first check that  for  $f$ to be well defined, we need $D_1KD_2\in[D_1HD_2].$  So we must find a $B\in GL(r,\C)$ such that $D_1KD_2= BD_1HD_2.$  Let $B=D_1AD_1^{-1}$. Then $D_1KD_2=D_1AHD_2=( D_1AD_1^{-1})D_1HD_2 = BD_1HD_2.$  So $D_1KD_2\in[D_1HD_2].$
 
The inverse map is  $f^{-1}([G])=[D_1^{-1}GD_2^{-1}]$.  Since both $f$ and $f^{-1}$ are quotients of continuous maps, $f$ and $f^{-1}$ are continuous. Therefore, $f$ is a homeomorphism.
\end{proof}

A consequence of Theorem \ref{equivrealspace} is that any phased matroid with a phirotope which is a rephasing of chirotope from an oriented matroid gives rise to a phased matroid whose realizability results defer to those of oriented matroids, including Mn\"ev's Universality Theorem (\ref{mnevs}), and that determining realizability is NP hard \cite{redbook}.  Proposition \ref{chirotopal} classifies such  phased matroids. 

Consider a matrix $M\in\RR^{r\times n}$. $M$ gives rise to a chirotope $\chi_M$, which is also a phirotope $\p_M$ with $\im(\p_M)\subseteq\{-1,0,+1\}$. We can  multiply $M$ by $A\in GL(r,\C)$ on the left to give  a matrix with non-real entries which is a realization of the same phased matroid as $M$. If $\det(A)=\alpha\notin\RR$, $\im(\p_{AM})\subseteq\{-\alpha,0,\alpha\}\not\subseteq\{-1,0,+1\}$, but $\p_{AM}$ is a phirotope of the same phased matroid that came from a real matrix. These special phirotopes are called \emph{complexified chirotopes}. We can also complexify chirotopes that do not come from a matrix. Let $\chi$ be a chirotope and $\beta\in S^1\backslash \RR$. Then $\beta\chi$ is a phirotope, but is not a chirotope. The set $\{\alpha\chi\mid\alpha\in S^1\}$ is a phased matroid. 
The phased matroids of such phirotopes are called \emph{complexified oriented matroids}, and have realization spaces that defer to the oriented matroid the original chirotope comes from. Similarly, as Theorem \ref{equivrealspace} suggests, rephases of  complexified chirotopes will give rise to a phase matroid $\M'$ whose realization space defers to the realiziation space of the original oriented matroid. 

It turns out the matrix $M$ from Example \ref{runex} is scaling equivalent to a real matrix. 

\begin{example}\label{runexscale}
Let \[A=\left(
 \begin{array}{ccc}
 e^{i\frac{\pi}{2}}&0&0\\
 0&2e^{i\frac{\pi}{4}}&0\\
 0&0&3
 \end{array} 
 \right)
 \text{ and } D=   \left(
\begin{array}{ccccccc}
     e^{i\frac{3\pi}{2}}& {}  0& 0& {0}&{0}  \\
 0 & {\frac{1}{2} e^{i\frac{7\pi}{4}}}  & 0& {0} &  { 0}   \\
 0 & {0}  & \frac{1}{3} & {0} &  {{0}} \\
  0 & {0}  & 0 & {{\frac{1}{2} e^{i\frac{7\pi}{4}}}} &  {{0}} \\ 
  0 & {0}  & 0 & {0} &  {{\frac{1}{3}}}  
\end{array}
\right),
\] then, 
 \[A\cdot\left(
\begin{array}{ccccccc}
    1& {}  0& 0& {1}&{1}  \\
 0 & {1}  & 0& {1} &  { 2}   \\
 0 & {0}  & 1 & {0} &  {{-1}}  
\end{array}
\right)\] gives rise to a complexified oriented matroid.  Furthermore,  \[A\cdot\left(
\begin{array}{ccccccc}
    1& {}  0& 0& {1}&{1}  \\
 0 & {1}  & 0& {1} &  { 2}   \\
 0 & {0}  & 1 & {0} &  {{-1}}  
\end{array}
\right)\cdot D=M.\] So,  $\M_M$ is a  rephrase of the a complexified oriented matroid. We call such phased matroids \emph{essentially oriented.} 
\end{example}
\begin{prop}[Essential orientability]\label{chirotopal} Let $\M$ be a  phased matroid. The following are equivalent:
\begin{enumerate}

\item $\M$ is a rephasing of  a complexified oriented matroid.
\item For some phirotope $\p$ of $\M$ and some $\alpha\in S^1$, there exists  $\rho\in(\S)^n$ such that $\im(\p^\rho)\subseteq\{-\alpha,0,\alpha\}$.
\item For every phirotope $\p_i$ of $\M$, there is an $\alpha_i\in S^1$ and  $\rho^i\in(\S)^n$ such that  $\im(\p_i^{\rho^i})\subseteq\{-\alpha_i,0,\alpha_i\}$.
\item There exist a phirotope $\p$ of $\M$ and  $\rho\in (\S)^n$ such that $\im(\p^\rho)\subseteq\{-1,0,1\}$.
\end{enumerate}

If $\M$ is realizable with realization space $\R(\M)$, the following are also equivalent to the above:  
\begin{enumerate} \setcounter{enumi}{4} 
  \item  For some $\rho\in(\S)^n$, there exists  $M\in\R(\M^\rho)$ with all real entries.
  \item  For any $M\in\R(\M),$  there exists a matrix $N \in \RR^{r\times n}$ such that $M$ and $N$ are scaling equivalent. 
\end{enumerate}\end{prop}

\begin{proof}\text{}

\begin{itemize}
\item[($1\Leftrightarrow 3$)]\[\begin{array}{rl} & \M \text{ is a rephasing of a complexified oriented matroid }\M_\chi \\
 \Leftrightarrow& \text{there is some }\rho\in(S^1)^n \text{ such that }\M^\rho=\M_\chi=\{\alpha\chi\mid\alpha\in S^1\} \\ 
\Leftrightarrow & \text{for all phirotopes }\p_i  \text{ of }\M,\text{ there is an }\alpha_i\in S^1\text{ such that }\p_i^\rho=\alpha_i\chi\\&\text{is a phirotope of }\M_\chi 
\\\Leftrightarrow& \im(\p_i^\rho)=\im(\alpha_i\chi)\subseteq\{-\alpha_1,0,\alpha_1\}.
\end{array}
\]

\item[($2\Leftrightarrow 3$)]  Let $\p,\p'$ be phirotopes of $\M$. Then there is some $\beta\in S^1$ such that $\beta\p=\p'$.  Suppose there is a $\rho\in(S^1)^n$ such that  $\im(\p^\rho)\subseteq\{-\alpha,0,\alpha\}$. Then $\im((\p')^\rho)=\im(\beta\p^\rho)=\beta\im(\p^\rho)=\{\beta(-\alpha),0,\beta\alpha\}.$
\item[($2\iff 4 $)] Suppose $\p$ is a phirotope of $\M$ and there is a $\rho$ such that $\im(\p^\rho)\subseteq\{-\alpha,0,\alpha\}$.  $\alpha^{-1}\p$ is a phirotope of $\M$ and $\im(\alpha^{-1}\p)=\alpha^{-1}\{-\alpha,0,\alpha\}=\{-1,0,1\}$.

\item[$(5\iff 6)$]
Suppose $\M$ is a rank $r$ realizable phased matroid on $n$ elements. For $\rho\in(S^1)^n$, there is a matrix $N\in\R(\M^\rho)$ such that $N\in\RR^{r\times n}$ if and only if there is $M\in \R(\M)$ such that $MD(\rho)=N$.  For every $M^i\in\R(\M)$, there is an $A^i\in GL(r,\C)$ such that $A^iM^i=M$. So $A^iM^iD(\rho)=N$  if and only if $A^iM^i$ and $N$ are scaling equivalent matrices for all $A^iM^i\in\R(\M)$.

\item[$(5\Rightarrow4)$] Suppose there is an $M\in\R(\M^\rho)$ such that $M\in\RR^{r\times n}$. Then every minor of $M\in\RR$. So $\p^\rho(\lambda)\in\{-1,0,+1\}$ for all $\lambda$. 
\item[$(4\Rightarrow 5)$] Suppose there is a phirotope $\p$ of $\M$, with $\p(1,\ldots, r)\ne0$, and $\rho\in (\S)^n$ such that $\im(\p^\rho)\subseteq\{-1,0,+1\}.$ Consider $(I|N)\in\R(\M^\rho)$. Since $\ph((I|N)_{i,j})=\p^\rho(1,\ldots,\hat i,\ldots, r,j)\in\{-1,0,+1\}$, every entry of $(I|N)\in\RR$. \qedhere\end{itemize}
\end{proof}

 \begin{dfn}[Essential orientablity]
 If any of the conditions of Proposition \hyperef{chirotopal}   hold, we say $\M$ (or any of its phirotopes) is  \emph{essentially oriented}.
\end{dfn}
Note that in previous literature, essentially oriented phased matroids and phirotopes are referred to as \emph{chirotopal} \cite{andersondel,BKR-G}.

\begin{cor}\label{chiro}
Let $\M$ be a rank $r$ essentially oriented phased matroid on $[n]$. Let $\M_\chi$ be a rank $r$ oriented matroid on $[n]$ with chirotope $\chi$ such that for some phirotope $\p$ of $\M$, $\rho\in (S^1)^n$, and for all $\lambda\in [n]^r, \p^\rho(\lambda)=\chi(\lambda)$. Then $\R(\M_\chi)\cong\R(\M)$.
\end{cor}
\begin{proof} 
This follows from Proposition \hyperef{chirotopal} and Theorem \hyperef{equivrealspace}.
\end{proof}
 
 As a consequence of Corollary \ref{chiro}, results about realizations of essentially oriented phased matroids defer to the results about their oriented matroid cousins. Our main results (Theorem \ref{uniformthm} and Theorem \ref{realizability})  in this paper, refer to phased matroids that are not essentially oriented. 

Now that we have determined that rephasing a phased matroid results is a phased matroid with a homeomorphically equivalent realization space, we can partition phased matroids into \emph{realization classes} and choose our favorite phirotope of our favorite phased matroid from the same realization class, which we will call the canonical phased matroid. We do this with the help of the {associated bipartite graph}, a tool also borrowed from matroid theory.

\begin{dfn}[Associated bipartite graph $G_{\MM}$ \cite{oxley}, 6.4]\label{abgdef}
Let $\MM$ be a rank $r$ simple matroid on $[n]$ such that $\{1,\ldots, r\}$ is a basis.  Let $G_{\MM}$ be the bipartite graph, with vertex set ${[n]}$, in which   $e_{i,j}$ is an edge if and only if $\{1,\ldots,\hat i, \ldots, r,j\} $ is a basis of $\MM$.  The two disjoint sets of vertices of the bipartite graph $ G_{\MM}$ are vertices $\{1,\ldots, r\}$, which are  the \emph{left hand side}, and vertices $\{r+1,\ldots,n\}$, which are the \emph{right hand side}.

If $\MM$ is realizable with realization $(I| N)$ then there is an edge $e_{i,j}$ in $G_{\MM}$ if and only if  $ (I| N)_{i,j}\ne 0$.
\end{dfn}

 The associated bipartite graph for the phased matroid $\M_M$ where $M$ is the matrix from Example \ref{runex} is shown in Figure \ref{ABGrunex}.

\begin{figure}[htb]
   \centering
  \def\svgwidth{100pt}
   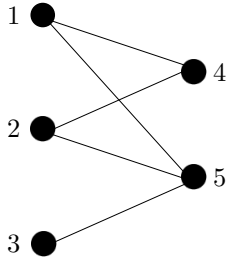 
   \caption{The associated bipartite graph for the $\M_M$.}
\label{ABGrunex}
\end{figure}

Notice that by Corollary \ref{phirophase}, if $\M$ is uniform, then $\G_\M$ is a complete graph. 
\begin{thm}[\cite{BryLuc}]\label{bipartite}
Let $\MM$ be a  rank $r$  matroid on $[n]$ realizable over $\mathbb F$ with $(I|N)\in \FF^{r\times n}$ a realization of $\MM$. Let $F=\{t_1,\ldots,t_{n-k}\}$ be a spanning
%\{e_{i_1,j_1},\ldots,e_{i_{n-1},j_{n-1}}\}$ be a spanning 
forest of $G_{\MM}$.   
Let $(s_1,\ldots,s_{n-k})$ be an ordered $n-k$ tuple of non-zero elements of ~$\FF$. 
Then $\MM$ has a unique realization  $(I| \widetilde N)\in\mathbb F^{r\times n} $ such that for each $i\in[n-k]$, the  entry of $\tilde N$ corresponding to $t_i$ is $s_i$.  

In fact, $(I| \widetilde N)$ can be obtained from $(I| N)$ by a sequence of row and column scalings. Hence $(I| \widetilde N)$ is scaling equivalent to $(I| N)$.

If $\MM$ is uniform, then $F$ is a tree. 
 \end{thm} 
As a result of Theorem \ref{bipartite}, we see that given a matrix, $(I| N)$,  whose associated bipartite graph has $k$ connected components, we can determine $n-k$ entries of $N$ to be any values we want. We will utilize this by choosing $(s_1,\ldots,s_{n-k})$ to be an $n-k$ tuple of ones. The next definition helps us decide which spanning tree we will use. With our choice of an $n-k$ tuple of ones, and a particular spanning tree, we can build an $r\times n$ array that might be a realization of $\MM$, if we can determine the real lengths of the entries that are not determined by the associated bipartite graph. 

Theorem \ref{bipartite} can be applied to phirotopes regardless of whether or not they are realizable. We will be able to use this and the tools to come to build an array, which will become a realization of the phirotope if and only of the phirotope is realizable. 
\begin{cor} Let $\phi$ be a rank $r$ phirptope $\phi:[n]^r\Rightarrow \S\cup \{0\}$. Let $F=\{t_1,\ldots,t_{n-k}\}$ be a spanning
forest of $G_{\MM_{\phi}}$.  Let $(s_1,\ldots,s_{n-k})$ be an ordered $n-k$ tuple of non-zero elements of ~$\FF$. Then there exists a $\rho$ such that $(-1)^{r-i}\p^{\rho}(1,\ldots,\hat i,\ldots,r,j)=s_l$ for all $e_{i,j}\in T$ and for each associated  $s_l\in S.$   \label{bipartphiro}
\end{cor}
The proof of this follows from Theorem \ref{bipartite} and Corollary \ref{phaseminor}.

\begin{dfn}[Canonical Spanning Tree] For each connected component in an associated bipartite graph $G$, with corresponding array $A$, we say the \emph{canonical spanning tree} of the component is the tree containing all edges associated to the first non-zero entry in each row of $M$, the last non-zero entry in each column of $A$ (assuming it does not make a cycle in $G$), and any other edges of $G$ corresponding to the last non-zero, entry in each column of $A$ that has not already been included, and does not create a cycle.
\end{dfn}

\begin{example} The canonical spanning tree for the Associated Bipartite Graph in Example \ref{ABGrunex} is seen as a collection of red dashed edges in Figure \ref{ABGrunexspan}.
\begin{figure}[htb]
   \centering
  \def\svgwidth{100pt} 
   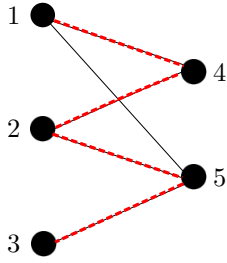 
   \caption{The associated bipartite graph for the $\M_M$.}
\label{ABGrunexspan}
\end{figure}

\end{example}
Therorem \ref{bipartite} is very useful to us because by Theorem \ref{equivrealspace}, scaling equivalent matrices give rise to phased matroids with  realization spaces with homeomorphically equivalent topology.  We will use this to distinguish a \emph{canonical realization} of a phased matroid that we will use to explore the topology of the realization space of any phased matroid in its equivalence class.

\begin{dfn}[Canonical form of a realization of $\M$]\label{canonicarealiz}
A realization $(I|\tilde N)$ of a rank $r$, uniform phased matroid  is  in \emph{canonical form} if the $n-k$ entries of $N$ that correspond to the canonical spanning tree of $G_\M$ are $1$.   A  uniform phirotope $\p$ of $\M$ is in \emph{canonical form} if  $\p(1,\ldots, r-1,j)=1$ for all $j\ge d$, and $\p(1,\dots, \hat i,\ldots, r+1)=(-1)^{i+r}$ for all $i<r$.
\end{dfn}

\begin{example}\label{runexcan} The realization of $\M_M$ from Example \ref{runex} which is in canonical form would look like the following (we will get to the unknown entry $*$ soon). 
\[\left(
\begin{array}{ccccccc}
    1& {}  0& 0& {1}&{*}  \\
 0 & {1}  & 0& {1} &  { 1}   \\
 0 & {0}  & 1 & {0} &  {{1}}  
\end{array}
\right)
\]
There may not be such a realization in $\R(\M_M)$ but there is such a realization in $\R(\M_M^\p)$ for  $\p=(e^{i\frac{3\pi}{2}},e^{i\frac{7\pi}{4}},1,e^{i\frac{7\pi}{4}},1)$. The values of $\p$ come from the diaganol matrix $D$ in Example \ref{runexscale}.
Since this matrix
% the matrix in Example  \ref{runexcan} 
is scaling equivalent to a real matrix, by Proposition \ref{chirotopal} the phased matroids $\M_M$ and $\M_M^\p$ are essentially oriented and the unknown value $*$ is real.
 
\end{example}
\begin{cor}\label{canonEO} Let $\M$ be a  rank $r$ realizable  phased matroid on $[n]$. Let  $\p$ be the phirotope of $\M$ in canonical form. Let $(I|\tilde N)$ be a realization of $\M$ in canonical form.  Then $\M$ is essentially oriented if and only if  $\im(\p)\subseteq\{-1,0+1\}$ and $(I|\tilde N) \in\RR^{r\times n}$. \end{cor}

\begin{proof} Suppose $\M$ is a realizable essentially oriented phased matroid with phirotope $\p$. 
Then $\im(\p)\in\{-\alpha,0,\alpha\}$ for $\alpha\in\S.$ If $\p$ is in canonical form, then $\p(1,\ldots, r)=1$. So $\alpha=1$ and $\im(\p)\in\{-1,0,+1\}$. In particular, $\p(1,\ldots,\hat i, \dots,d,j)=\ph((I|\tilde N)_{i,j})\in\{-1,0,+1\}$, so $(I|\tilde N)\in\RR^{r\times n}$. 
\end{proof}

 The canonical form of a realization of  a  realizable uniform phased matroid $\M$ is\[ (I|\tilde{N})=\left(\begin{array}{c|c}{I} & \begin{array}{c}1\\\vdots\\1\end{array}
 \begin{array}{c}
\begin{array}{|cc}&N\\&\\\hline \end{array}\\
\\\begin{array}{cc}\cdots & 1\end{array}
\end{array}
\end{array}\right).
\]

\begin{cor}\label{realspacecanonical}Let $\M$ be a  uniform rank $r$ phased matroid on $[n]$. There exists $\rho\in(\S)^{n}$ such that
 $$\R(\M)\cong\{(I|\tilde N)\mid(I|\tilde N) \text{ is  a  realization of } \M^\rho\text{ in canonical form}\}\times\Rp^{n-1}.$$ 
\end{cor}

\begin{proof} The existence of $(I|\tilde N)$  follows from Theorem \ref{bipartite}. Using Theorem \ref{bipartite}, in building canonical realization $(I|\tilde N)$, we chose the values associated with each of the $n-1$ edges of our conanical spanning tree to be $1$.  It is clear that if we had chosen any other set of positive real values  $(s_1,\ldots,s_{n-1})\in\Rp^{n-1}$, the result would be a different matrix in $ \R(\M^{\rho})$. This $n-1$ degrees of freedom  is where the $\Rp^{n-1}$ comes from. %%%%%%%%%
The equivalence follows from Theorem \ref{equivrealspace}.\end{proof}

\subsection{Triangle Lemma}\label{dellemma}
  \begin{dfn}[Triangular equation]\label{triangulareqn}
   Let $\alpha, \beta, \gamma\in\S$ such that  $\alpha\ne\pm\beta$. An equation of the form $\gamma=\ph(s_1\alpha-s_2\beta)$,  where $s_1,s_2\in\Rp,$  is called a \emph{triangular equation.}
\end{dfn}

The following Lemma will be used often to prove later results. 
   \begin{lem}[Triangle Lemma]\label{triangle} If $\gamma=\ph(s_1\alpha-s_2\beta)$ is a triangular equation and   either $s_1$ or $s_2$ are known, then the equation can be solved uniquely for the unknown quantity. 
   \end{lem}
 
   \begin{proof}
   Consider $\gamma, \alpha, $ and $\beta\in\S$. Consider the rays through $\alpha, ~\beta, $ and $\gamma$ in the complex plane. Since $\gamma=\ph(s_1\alpha-s_2\beta)$, we can draw a triangle with interior angles $\theta $ and $\psi$ as seen in Figure \hyperef{trilem}. The angles  $\theta$ and $\psi$ are determined by $\alpha,~\beta,$ and $\gamma$.  Therefore, the triangular equation determines the triangle up to similarity. Since either $s_1$ or $s_2$ is known, the entire triangle is determined by the triangular equation. 
 \begin{figure}[htb]
  \centering
  \def\svgwidth{180pt}
  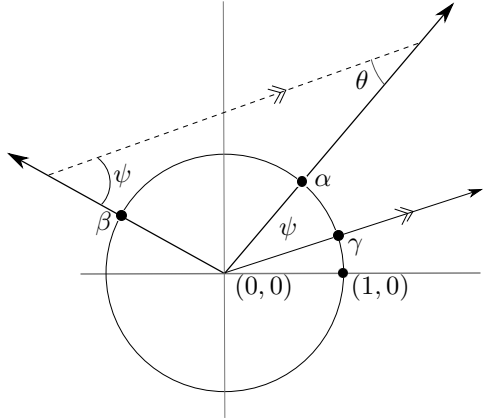
  \caption{The triangle with interior angles $\theta$ and $\psi$ is constructed from the equation $\gamma=\ph(s_1\alpha-s_2\beta)$.  }\label{trilem}
\end{figure}
   \end{proof}

%%%%%%%%%%%%%%%%%%%%%%%%%%%%%%%%%%%%%%%%%%%%%%%%%%%%%%%%%%%%%%%%%%%%%%%%%%%%%%%%%%%%%%%%%%%%%%%%%%%%%%%%%%%%%%%%%%%%%%%%%%%%%%%%%%%%%%%%%%%%%%%%%%%%%%%%%%%%%%%%%%%%%%%%%%%%%%%%%%%%%%%%%%%%%%%%%%%%%%%%%%%%%%%%%%%%%%%%%%%%%%%%%%%%%%%%%%%%%%%%%%%%%%%%%%%%%%%%%%%%%%%%%%%%%%%%%%%%%%%%%%%%%%%%%%%%%%%%%%%%%%%%%%%%%%%%%%%%%%%%%%%%%%%%%%%%%%%%%%%%%%%%%%%%%%%%%%%%%%%%%%%%%%%%%%%%%%%%%%%%%%%%%%%%%%%%%%%%%%%%%%%%%%%%%%%%%%%%%%%%%%%%%%%%%%%%%%%%%%%%%%%%%%%%%%%%%%%%%%%%%%%%%%%%%%%%%%%%%%%%%%%%%%%%%%%%%%%%%%%%%%%%%%%%%%%%%%%%%%%%%%%%%%%%%%%%%%%%%%%%%%%%%%%%%%%%%%%%%%%%%%%%%%%%%%%%%%%%%%%%%%%%%%%%%%%%%%%%%%%%%%%%%%%%%%%%%%%%%%%%%%%%%%%%%%%%%%%%%%%%%%%%%%

\section{Proof of Theorem \ref{uniformthm}}\label{sec:proof}
Now that we have a canonical realization and the Triangle Lemma, we can prove  our main Theorem:
\begin{thm}[Therorem \ref{uniformthm}]
Let $\M$ be a rank $r$, uniform, not essentially oriented, realizable phased matroid on $[n]$. Then $\R(\M)\cong\Rp^{n-1}$. \end{thm}

As a warm-up, we will prove the case for rank 2 phased matroids, which is stated as Lemma \ref{rigid}. 

\subsection{Rank 2 phased matroids}

For rank 2 phased matroids we can drop the condition that the phased matroid is uniform and not essentially oriented and get a similar result.

Recall, Lemma \ref{rigid} states that for a
 simple, rank $2$, realizable phased matroid $\M$ on $n$ elements,  $\R(\M)\cong\Rp^{n-1+k}$ where  $k=0$ if $\M$ is not essentially oriented and $k>0$ if $\M$ is essentially oriented. 

{This result, for the uniform, rank $2$   phased matroids that are not essentially oriented,  was previously proven in \cite{BKR-G} using cross ratios. Since cross ratios do not generalize to higher dimensions, their proof is not generalizable to phased matroids with  rank $> 2$. For essentially oriented phased matroids the Lemma follows from the fact that all rank-2 oriented matroids have contractible realization spaces \cite{redbook}. 
 }
\subsubsection{Proof of Lemma \ref{rigid}}
\begin{proof} 

 If $\M$ is a rank $2$ phased matroid and is  not uniform,  then either zeros appear as entries of columns $3,\ldots, n$ in the canonical realization or a pair of elements $\{i,j\}$ are parallel.  But any column with a zero  is  either  a loop in $\M$, or is parallel to either $e_1$ or $e_2$. If $e$ is a loop of of $\M$ then $\R(\M/\{i\})\cong\R(\M)$ and if elements $i,j$ are parallel in $\M$ then   $\R(\M\backslash\{i\})\times\Rp\cong\R(\M)$ \cite{mydissertation}. For the case when $\M$ is essentially oriented, the proof follows from analogous results about oriented matroids. So without loss of generality, we may assume $\M$ is a uniform, not essentially oriented, realizable phased matroid. 
 
% So, without loss of generality we may assume $\M$ is a uniform, realizable, phased matroid. 
%Suppose $\M$ is essentially oriented. Then by Corrolary \ref{canonEO}, the canonical realization $(I|\tilde N)\in\RR^{r\times n}$.

Consider a canonical realization $$M=\left(\begin{matrix}{1} &{0} &1&  {s_4}{\alpha_4}&\cdots&{s_n}{\alpha_n}\\{0} &{1} &{1}&{1} & \cdots &{1}\end{matrix}\right)$$
of $\M^\rho$. 

Notice that $\p(1,i)$ and $\p(2,j)$  give us the phases of each entry of $N$ so the only unknown information about the realization is the values of $s_4,\ldots,s_n\in\Rp$. In fact 

$\{(I|\tilde N)\mid(I|\tilde N) \text{ is  a canonical realization of } \M^\rho\}$ can be thought of as \begin{equation}\label{point}\{(s_4,\ldots,s_n)\in\Rp^{n-4}\mid s_i=m_{1,i}\mbox{ for  a realization }M\mbox{ of }\M^\rho\}.\end{equation}

Since  $\M$ is not essentially oriented,  for some $j\in\{4,\ldots,n\},~\alpha_j\notin\{-1,+1\}.$  Therefore, by Lemma \ref{triangle}, $s_j$ is determined by the equation $\p(3,j)=\ph(1-s_j\alpha_j)$.  For any other $k\ne j\in\{4,\ldots,n\}$, the equation $\p(k,j)=\ph(s_k\alpha_k-s_j\alpha_j)$ determines $s_k$. 

Since $s_j$ is determined by $\p$ for all $j\in\{4,\ldots,n\}$, the set $\{(I|\tilde N)\mid(I|\tilde N) \text{ is  a canonical realization of } \M^\rho\}$ is a single point. Therefore, by Corollary \ref{realspacecanonical}, $\R(\M)=\Rp^{n-1}$.
%If $\M$ is essentially oriented, then $\alpha_4,\ldots,\alpha_n\in\{-1,+1\}.$  Furthermore, $\p(3,j)=\ph(1-s_j\alpha_j)$ restricts $s_j$ to an open interval $R\cong\Rp$. So, $k=n-4$ and $\R(\M)\cong\Rp^{n-4}\times\Rp^{n-1}=\Rp^{2n-5}.$
\end{proof}

\subsubsection{Proof of Theorem \ref{uniformthm}}

The previous proof  for rank 2 phased matroids provides insight into the proof for rank $r$ uniform phased matroids that are not essentially oriented.
\begin{proof}
 Consider $(I|\tilde N)\in\R(\M)$. Since $\M$ is uniform, all entries of $\tilde N$ are non-zero. Since $\M$ is not essentially oriented, there is at least one  non-real entry in $\tilde N$.

Furthermore, by Corollary \hyperef{phaseminor}, for each $1\le i\le r-1$ and $r+1\le j\le n$, $\alpha_{i,j}$ is determined by $\p$, that is $\alpha_{i,j}=(-1)^{r-i}\p(1,\ldots,i-1,j,i+1,\ldots,n)$. 
Thus, by Corollary \hyperef{realspacecanonical}, it remains to find all $s_{i,j}$ such that 
$1\le i\le r-1,~r+1\le j\le n$. 
 {\begin{displaymath}
(I| \tilde N)=\left(\begin{array}{cccc|cccc}
1&0&\cdots&0&1 & s_{1,d+2}\alpha_{1,d+2} & \cdots & s_{1,n}\alpha_{1,n} \\ 
0&\ddots&0&\vdots&\vdots & \vdots & \ddots & \vdots \\ 
\vdots&0&\ddots&0&1 & s_{r-1,d+2}\alpha_{r-1,r+2} & \cdots & s_{r-1,n}\alpha_{r-1,n} \\ 
0&\cdots&0&1&1 & 1 & \cdots & 1
                \end{array}\right)\in\R(\M).
                \end{displaymath}
}
  Notice, by Lemma \ref{minorsign}, for $H=\{i,m\}$ where $i\le i\le m\le r$, $\sigma= \sum_{k\notin H, k>i}1+\sum_{k\notin H, k>m}1=r-i-1+r-m=2r-i-m-1$, which has the same parity of $i+m+1$. Therefore, 
 \begin{align*}(-1)^{i+m+1}[1,\ldots,\hat i,\ldots,\hat m,\ldots, r,k,j]_{(I|\tilde N)}=\det\left(\begin{array}{cc}
 s_{i,k}\alpha_{i,k}&s_{i,j}\alpha_{i,j}\\
 s_{m,k}\alpha_{m,k}&s_{m,j}\alpha_{m,j}
                \end{array}\right).\end{align*}To determine the value of $s_{i,j}$, there are four cases to consider. In each case, we find a $\lambda$-minor of $(I|\tilde N)$ that is equal to  a $2\times 2$ minor of $(I|\tilde N)$ which results in a triangular equation in which $s_{i,j}$ is the only unknown. 
 
 For Case 1, suppose $\alpha_{i,j}\ne\pm1$. Then  
 \begin{equation}(-1)^{r-i+1}\p(1,\ldots,\hat{i},\ldots, r-1,j)=\ph\left(\det\left(\left|\begin{array}{cc}1 & s_{ij}\alpha_{ij} \\1 & 1\end{array}\right|\right)\right)=
 \ph(1-s_{i,j}\alpha_{i,j})\end{equation} 
 is a triangular equation and $s_{i,j}$ is determined by the Triangle Lemma.  
 
 For the remaining three cases, we assume $\alpha_{i,j}=\pm1$. 
 
  For the second case, suppose  there exists  $k>r+1$ such that $\alpha_{i,k}\ne\pm1$. We know $s_{ik}$ from Case $1$. Without loss of generality, assume $k>j$.  Then 
   \begin{align*}
  (-1)^{r-i+1}\p(1,\ldots,\hat{i},\ldots, r-1,j,k)=&\\
  \ph\left(\det\left(\left|\begin{array}{cc} s_{ij}\alpha_{ij} &s_{i,k}\alpha_{i,k}\\1 & 1\end{array}\right|\right)\right)=&
\ph(s_{i,j}\alpha_{i,j}-s_{i,k}\alpha_{i,k})
\end{align*}  is a triangular equation and $s_{i,j}$ is determined by the Triangle Lemma.
 
 For the final two cases, $s_{i,k}\alpha_{i,k}\in\RR$ for all $r+1<k\le n$. This means the entire row has phase $\pm1$. 
 
 For Case 3, we assume there is a non-real entry in the $j^{th} $ column.  
  The equation  
 \begin{align*}(-1)^{i+m+1}\p(1,\ldots,\hat{i}\ldots,\hat{m}\ldots, r+1,j)=&\\
 \ph\left(\det\left(\begin{array}{cc}
 1 & s_{i,j}\alpha_{i,j} \\
 1 & s_{m,j}\alpha_{m,j} \end{array}\right)\right) 
 =&\ph(s_{m,j}\alpha_{m,j}-s_{i,j}\alpha_{i,j})\end{align*}
 is triangular and $s_{i,j}$ is determined by the Triangle Lemma.  
 
 For the final case, assume all entries in row $i$ and column $j$ are real. There must be a non-real entry $s_{m,k}\alpha_{m,k}\notin\RR$ with $m\ne i, ~k\ne j$.  From previous cases, $s_{m,j}$ and $s_{i,k}$ are determined.  Also, we know $\alpha_{i,j},~\alpha_{i,k},$ and $\alpha_{m,j}$ are all $\pm1$.
 So, \begin{align*}(-1)^{i+m+1}\p(1,\ldots,\hat{i}\ldots,\hat{m}\ldots,k,j)=&\\\ph\left(\det\left(\begin{array}{cc}s_{i,k}(\pm1) & s_{i,j}(\pm1) \\s_{m,k}\alpha_{m,k} & s_{m,j}(\pm1) \end{array}\right)\right) =&\ph(s_{i,k}s_{m,j}-s_{i,j}s_{m,k}\alpha_{m,k})\end{align*} is a triangular equation. So $s_{i,j}$ is determined by the Triangle Lemma.

All $s_i$'s are determined. So\[
\left\{
 \left(\begin{array}{c|c}{I} & \begin{array}{c}1\\\vdots\\1\end{array}
 \begin{array}{c}
\begin{array}{|cc}&N\\&\\\hline \end{array}\\
\\\begin{array}{cc}\cdots & 1\end{array}
\end{array}
\end{array}\right)\vline
\left(\begin{array}{c|c}{I} & \begin{array}{c}1\\\vdots\\1\end{array}
 \begin{array}{c}
\begin{array}{|cc}&N\\&\\\hline \end{array}\\
\\\begin{array}{cc}\cdots & 1\end{array}
\end{array}
\end{array}\right)
\text{ is a realiz. of } \M\right\}\cong(\Rp)^0.
 \]  Therefore,  $\R(\M)=(\Rp)^0\times(\Rp)^{n-1}\cong (\Rp)^{n-1}$.
 \end{proof}
%%%%%%%%%%%%%%%%%%%%%%%%%%%%%%%%%%%%%%%%%%%%%%%%%%%%%%%%%%%%%%%%%%%%%%%%%%%%%%%%%%%%%%%%%%%%%%%%%%%%%%%%%%%%%%%%%%%%%%%%%%%%%%%%%%%%%%%%%%%%%%%%%%%%%%%%%%%%%%%%%%%%%%%%%%%%%%%%%%%%%%%%%%%%%%%%%%%%%%%%%%%%%%%%%%%%%%%%%%%%%%%%%%%%%%%%%%%%%%%%%%%%%%%%%%%%%%%%%%%%%%%%%%%%%%%%%%%%%%%%%%%%%%%%%%%%%%%%%%%%%%%%%%%%%%%%%%%%%%%%%%%%%%%%%%%%%%%%%%%%%%%%%%%%%%%%%%%%%%%%%%%%%%%%%%%%%%%%%%%%%%%%%%%%%%%%%%%%%%%%%%%%%%%%%%%%%%%%%%%%%%%%%%%%%%%%%%%%%%%%%%%%%%%%%%%%%%%%%%%%%%%%%%%%%%%%%%%%%%%%%%%%%%%%%%%%%%%%%%%%%%%%%%%%%%%%%%%%%%%%%%%%%%%%%%%%%%%%%%%%%%%%%%%%%%%%%%%%%%%%%%%%%%%%%%%%%%%%%%%%%%%%%%%%%%%%%%%%%%%%%%%%%%%%%%%%%%%%%%%%%%%%%%%%%%%%%%%%%%%%%%%%%%

 \section{Realizability criterium} All rank 2 oriented matroids are realizable. In contrast, not all phased matroids are realizable. In this section we give an example of a non-realizable phased matroid, and a simple realizability criterium that can be used to determine in a rank-2, or any other non-essentially oriented, uniform phased matroid is realizable.
\subsection{A non-realizable rank 2 phased matroid}
\begin{example}\label{nonreal2ex}
The phased matroid $\M$ with phirotope $\p$ such that $\p(1,2)=\p(1,3)=\p(1,4)=\p(1,5)=-\p(2,3)=1$, $\p(2,4)=-e^{i\frac{\pi}{2}}$, $\p(2,5)=-e^{i\frac{\pi}{3}},\p(3,4)=e^{i\frac{7\pi}{4}},\p(3,5)=e^{i\frac{5\pi}{3}},$ and $\p(4,5)=e^{i\frac{5\pi}{6}}$ is not realizable. 
\end{example}

\begin{proof}
 It is not hard to check the  combinatorial complex Grassman-Pl\"ucker relations to confirm $\p$ is a phirotope  of a   phased matroid.  Suppose $\M$ is realizable. Since $\p$ is in canonical form, the proof of Lemma \hyperef{rigid}  provides a construction of a potential realization of $\p$. 
 
 The values of $\p$ on all pairs except $\{4,5\}$ determine the following canonical realization of  $\M$: \[M=\left(\begin{array}{ccccc}
 1&0&1&e^{i\frac{\pi}{2}}&e^{i\frac{\pi}{3}}\\
 0&1&1&1&1\end{array}\right),\]  where $\p(2,4)$ and $\p(2,5)$ determine the phases of $M_{1,4}$ and $M_{1,5}$ respectively, and the norm of each entry is determined by  $\p(3,4)$ and $\p(3,5)$. But $\p(4,5)=e^{i\frac{\pi}{6}}\ne\ph(e^{i\frac{\pi}{2}}-e^{i\frac{\pi}{3}})=e^{i\frac{11\pi}{12}}$. So $\M$ is not realizable.  
 \end{proof}
 
 In \cite{BKR-G},  a different method is provided to test the realizability of the above phased matroid in which a phirotope is confirmed to be realizable if an equation of 24 terms sums to 0. Their method can be generalized to uniform rank 2 phased matroids on $n$ elements by checking the sum for all  $n\choose 5$ subsets of elements of the groundset.

  Example \ref{nonreal2ex} sheds light on a \emph{two point realizability criteria} for rank $2$ phased matroids. 
 
 \begin{prop}\label{rank2realizability}Let $\M$ be a uniform, rank $2$, not essentially oriented phased matroid on $[n]$ with phirotope $\p$ in canonical form. Let $\theta_j=\arg(\p(3,j))$ and $\psi_j=\arg(\p(3,j))-\arg(\p(2,j))$. Then $\M$ is realizable if and only if for any pair $j,k\in\{4,\ldots,n\}$, 
 % such that $\p(2,j)\p(2,k)\ne\pm1$, 
  \begin{equation}\label{nonreal2}\p(j,k)=\ph\left(\frac{\sin(\theta_k)}{\sin(\psi_k)}\p(2,k)-\frac{\sin(\theta_j)}{\sin(\psi_j)}\p(2,j)\right).\end{equation}
\end{prop}

\begin{proof} If $\M$ is realizable, a unique canonical realization $M$ of $\M$ can be constructed using $\p(2,j)$  and $\p(3,k)$ for all $j,k>3$ where $\frac{\sin(\theta_j)}{\sin(\psi_j)}$ is the norm of $M_{1,j}$. Equation \hyperef{nonreal2} is  $\p(j,k)=\ph([j,k]_M)$ and must hold by definition of realizability.  
\end{proof}

The (5-point) realizability criteria for rank 2 uniform phased matroids  is given in \cite{BKR-G} depends on cross ratios, so it is not  generalizable to phased matroids with rank greater than 2. However, using the Triangle Lemma as its foundation, the realizability criteria from Proposition \ref{rank2realizability} is easily extended to higher ranked phased matroids. 

\begin{thm}\label{realizability}Let $\M$ be a rank $r$, uniform, not essentially oriented phased matroid on $[n]$ with canonical phirotope $\p$. 
Let $(I|N)$ be a matrix such that $\ph((I|N)_{i,j})=\p(\hat i,j)$ and $|((I|N)_{i,j})|$ is determined by $\p(\hat{\{ i,r\}},r+1,j)$ as in Theorem \hyperef{uniformthm}. 
$\M$ is realizable if and only if for all $\lambda\in[n]^r,~\p(\lambda)=\ph([\lambda]_{(I|N)})$.
\end{thm}

\begin{proof}Given a uniform phased matroid with a phirotope $\p$ in canonical form, it is always possible to construct a potential canonical realization $(I|N)$ of $\M$ following the construction in the proof of Theorem \hyperef{uniformthm}.  If $~\p(\lambda)=\ph([\lambda]_{(I|N)})$ for all $\lambda\in[n]^r,$ then $(I|N)\in\R(\M)$. Otherwise, $\M$ is not realizable.
 \end{proof}

 \section{Conclusion}
 We have shown that in comparison to oriented matroids, uniform phased matroids can have remarkable simple realization space, and we can answer the realizability question for uniform not essentially oriented phased matroids in polynomial time.  
% Some non-uniform phased matroids also have these really nice realization spaces, but we show others that do not, but rather have realizations spaces described by Mn\"ev's universality theorem \ref{mnevs}. In fact, we construct 4 classes of non-uniform, not essentially oriented phased matroids with this property.  We also showed that all simple connected rank 3 phased matroids have a computable realization space.  
 In light of a new umbrella theory that encompasses oriented and phased matroids called $F$-matroids, in our future work, we hope to classify  which other hyperfields contain the important property described by  the triangle lemma. 

%%%%%%%%%%%%%%%%%%%%%%%%%%%%%%%%%%%%%%%%%%%%%%%%%%%%%%%%%%%%%%%%%%%%%%%%%%%%%%%%%%%%%%%%%%%%%%%%%%%%%%%%%%%%%%%%%%%%%%%%%%%%%%%%%%%%%%%%%%%%%%%%%%%%%%%%%%%%%%%%%%%%%%%%%%%%%%%%%%%%%%%%%%%%%%%%%%%%%%%%%%%%%%%%%%%%%%%%%%%%%%%%%%%%%%%%%%%%%%%%%%%%%%%%%%%%%%%%%%%%%%%%%%%%%%%%%%%%%%%%%%%%%%%%%%%%%%%%%%%%%%%%%%%%%%%%%%%%%%%%%%%%%%%%%%%%%%%%%%%%%%%%%%%%%%%%%%%%%%%%%%%%%%%%%%%%%%%%%%%%%%%%%%%%%%%%%%%%%%%%%%%%%%%%%%%%%%%%%%%%%%%%%%%%%%%%%%%%%%%%%%%%%%%%%%%%%%%%%%%%%%%%%%%%%%%%%%%%%%%%%%%%%%%%%%%%%%%%%%%%%%%%%%%%%%%%%%%%%%%%%%%%%%%%%%%%%%%%%%%%%%%%%%%%%%%%%%%%%%%%%%%%%%%%%%%%%%%%%%%%%%%%%%%%%%%%%%%%%%%%%%%%%%%%%%%%%%%%%%%%%%%%%%%%%%%%%%%%%%%%%%%%%%

 \fontsize{12}{12pt} \selectfont
\bibliographystyle{amsplain}  
\bibliography{references-1.bib}

\end{document}